\documentclass[10pt]{IEEEtran}
\usepackage[english]{babel}

\usepackage{amsmath}
\usepackage{amsfonts}
\usepackage{amsthm}
\usepackage{amscd,latexsym,amssymb}
\usepackage{graphicx}
\usepackage{url}
\usepackage{color}

\newtheorem{rem}{Remark}
\newtheorem{thm}{Theorem}
\newtheorem{prop}{Proposition}
\newtheorem{cor}{Corollary}
\newtheorem{lem}{Lemma}

\begin{document}

\title{
Locating dominating codes: Bounds and extremal cardinalities
 \thanks{This work was supported in part by grants PAI FQM-305, P06-FQM-01649, Gen. Cat. DGR 2009SGR1040, 2009SGR1387, MTM2011-28800-C02-01, MTM2009-07242. It is also partially supported by the ESF EUROCORES programme EuroGIGA - ComPoSe IP04 - MICINN Project EUI-EURC-2011-4306. }}

\author{
Jos\'e C\'aceres, \thanks{J. C\'aceres and M.L. Puertas are with the Statistics and Applied Math Department in the Universidad de Almer\'{i}a in Almer\'{i}a, Spain (e-mail:\{jcaceres,mpuertas\}@ual.es)}
\and Carmen Hernando, \thanks{C. Hernando, M. Mora and I.M. Pelayo are with  Applied Math I, II and III Departments respectively, in the Universitat Polit\`ecnica de Catalunya in Barcelona, Spain (e-mail:\{carmen.hernando,merce.mora,ignacio.m.pelayo\}@upc.edu)}
\and Merc\`e Mora,
\and Ignacio M. Pelayo and
\and Mar\'{\i}a Luz Puertas}%

\maketitle

\begin{abstract}
In this work, two types of codes such that they both dominate and locate the vertices of a graph are studied. Those codes might be sets of detectors in a network or processors controlling a system whose set of responses should determine a malfunctioning processor or an intruder. Here, we present our more significant contributions on $\lambda$-codes and $\eta$-codes concerning concerning bounds, extremal values and realization theorems.
\end{abstract}


\begin{IEEEkeywords}
Network problems, Graph theory, codes on graphs, covering codes, locating dominating codes.
\end{IEEEkeywords}

\section{Introduction}
Assume a building floor modeled as a graph. Minimum locating-dominating sets or $\lambda$-codes can be used to determine the exact location of an object in the graph provided that the object cannot occupy the same vertex as the detection device, for instance a fire alarm placed on a wall or the ceiling. Each alarm sends a signal when detecting a fire in any of its adjacent vertices and the activated signals will univocally determine the place of the fire. Thus locating-dominating sets are also covering codes.

Very often, the detection range of the device is not as limited as in the previous example; imagine a surveillance camera instead of a fire alarm. Here the detector gives the distance to the object, say an intruder, and the set of distances unambiguously locates the object. However, in order to prevent failures and maintain the properties of a covering, or dominating, code it will be interesting to ensure that any location under surveillance is next to at least one camera (perhaps for identifying the intruder). Then we have a minimum metric-locating-dominating set or $\eta$-code.

Another application comes from multiprocessor architecture. Here each vertex corresponds to a processor and each edge to a dedicated link between two processors and some processors have the task to check the rest of the system. Clearly, any complete set of outputs should determine a faulty processor. If the processors are only able to check its immediate neighbors then a $\lambda$-code is necessary. On the other hand, if the checking processors have an unlimited range of action within the system then the best covering code that one could use is a $\eta$-code. Whenever the processors should be checked by themselves, another popular class of codes such that identifying codes are necessary~\cite{hokale06,kachale,latr08}. However, the existence of those codes is not guaranteed for any graph, and then a locating-dominating code is the next best alternative. A complete list of continuously updated papers involving different kinds of codes is to be found in~\cite{lobstein}.

The immediate problem here is to determine the minimum number of detectors needed for each code. It is also interesting to know some trade-offs between using $\lambda$-codes and $\eta$-codes. The rest of the paper is organized as follows: In Section~\ref{definiciones}, the main concepts and definitions are presented. In Section~\ref{cotas}, tight bounds of both parameters are given. Those cases with extremal cardinalities of $\lambda$ and $\eta$ are discussed in Section~\ref{valoresextremos}, and in Section~\ref{realizabilidad} several realization theorems for any possible values are provided.

\section{Formal definitions and related work}
\label{definiciones}

All the graphs considered are  finite, undirected, simple, and connected. Given a graph $G=(V,E)$, the \emph{open neighborhood} of a vertex $v\in V$ is $N(v)=\{u\in V | uv\in E\}$ and its \emph{degree} $deg(v)=|N(v)|$. The distance between two vertices $v$ and $w$ is denoted by $d(v,w)$ and the diameter $diam(G)$ is the maximum distance within two vertices of $G$. For undefined basic concepts we refer the reader to introductory graph theoretical literature, e.g.,~\cite{chartrand}.

This work relies on two main concepts such are domination and location. Thus, a set $D\subseteq V$ is \emph{dominating} if for every vertex $v\in V\setminus D$, $N(v)\cap D\neq\emptyset$. The \emph{domination number} $\gamma(G)$ is the minimum cardinality of a dominating set of $G$ and a dominating set of cardinality $\gamma(G)$ is called a \emph{$\gamma$-code} ~\cite{hahesl}. On the other hand, let $S=\{x_1,\ldots,x_k\}$ be a subset of vertices. For any $v\in V\setminus S$, the vector of \emph{metric coordinates} of $v$ with respect to $S$ is the ordered $k$-tuple $c_S(v)=(d(v,x_1),\ldots,d(v,x_k))$. The set $S$  will be \emph{locating} if for every pair of distinct vertices $u,v\in V$,
$c_S(u)\neq c_S(v)$. The \emph{metric dimension} $\beta(G)$ is the minimum cardinality of a locating set of $G$~\cite{hame,slater75}. A locating set of cardinality $\beta(G)$ is called a \emph{metric code}.

Undoubtedly, it is of interest for a code to be dominating and locating, and there exist several ways to define it. For instance, a \emph{metric-locating-dominating set} is directly a dominating and locating set. The \emph{metric-location-domination number} $\eta(G)$ is the minimum cardinality of a metric-locating-dominating set of $G$ and a metric-locating-dominating set of cardinality $\eta(G)$ is called an \emph{$\eta$-code}~\cite{heoe}. A different and more restrictive definition is the following: a set $D\subseteq V$ is a \emph{locating-dominating set} if every two vertices  $u,v\in V(G)\setminus D$ verify that $\emptyset\neq N(u)\cap D\neq N(v)\cap D\neq\emptyset.$ The \emph{location-domination number}  $\lambda(G)$ is the  minimum cardinality of a locating-dominating set.  A locating-dominating set of cardinality $\lambda(G)$ is called a \emph{$\lambda$-code}~\cite{slater87,slater88}.

Certainly, every locating-dominating set is both locating and dominating. However, a set which locates and dominates is not necessarily a locating-dominating set. For example, consider the path $P_6$ with vertices $\{0,..,5\}$. Then the set $D=\{1,4\}$ is both dominating and locating, but it is not a locating-dominating set since  $N(3)\cap D=N(5)\cap D=\{4\}$.

Location and domination are hereditary properties. Particularly, if for two subsets $S_1,S_2\subset V$ the set $S_1$ is locating and $S_2$ is dominating, then $S_1\cup S_2$ is both locating and dominating.

A straightforward consequence of the above definitions follows:

\begin{prop}\label{ineq} For every graph $G$,
$\max\{\gamma(G),\beta(G)\} \leq \eta(G) \leq \min\{\gamma(G)+\beta(G),\lambda(G)\}$
\end{prop}

\begin{table}[t]
\begin{center}\begin{tabular}{|c||c|c|c|c|} \hline
$G$                       & $\gamma$   &  $\beta$   &  $\eta$   &  $\lambda$    \\  \hline\hline
$P_n$, $n>3$     &  $\lceil\frac{n}{3}\rceil$     & 1 & $\lceil\frac{n}{3}\rceil$ & $\lceil\frac{2n}{5}\rceil$  \\ \hline
$C_n$, $n>6$    & $\lceil\frac{n}{3}\rceil$  &  2  &  $\lceil\frac{n}{3}\rceil$ &  $\lceil\frac{2n}{5}\rceil$ \\ \hline
$K_n$, $n>1$                     &      1     &    $n-1$   &    $n-1$   &  $n-1$       \\ \hline
$K_{1,n-1}$, $n>2$        &      1     &    $n-2$   &    $n-1$   &  $n-1$        \\ \hline
$K_{r,n-r}$, $1<r\le n-r$ &      2     &    $n-2$   &    $n-2$   &  $n-2$      \\ \hline
$W_{1,n-1}$, $n>7$   &  1  &  $\lfloor\frac{2n}{5}\rfloor$ & $\lceil\frac{2n-2}{5}\rceil$ & $\lceil\frac{2n-2}{5}\rceil$  \\ \hline
\end{tabular}\end{center}
\caption{Domination and location parameters of some basic families.}
\label{tabla}       
\end{table}

In the rest of this paper,  $P_n$, $C_n$ and $K_n$ denote the path, cycle and complete graph of order $n$, respectively. In all cases, unless otherwise stated, the set of vertices is $\{0,1,\cdots,n-1\}$. In addition, $K_{p,n-p}$ and $W_{1,n-1}$  denote the complete bipartite graph (being its smallest stable set of order $p$)  and the wheel of order $n$. Check the values of the domination and location parameters for those families of graphs in Table~\ref{tabla}.

Finally, the strong grid $P_n \boxtimes P_m $ has as vertices the pairs of integers $(i,j)$ such that $0\leq i\leq n-1, 0\leq j\leq m-1$, and two vertices $(i,j)$ and $(i',j')$ are adjacent when $|i-i'|\leq 1$ and $|j-j'|\leq 1$. This operation can be iterated to obtain the $k$-dimensional strong grid $P_{n_1}\boxtimes \dots \boxtimes P_{n_k}$ \cite{hik11}.

\section{Bounds}\label{cotas}

In this section we will bound the values of $\eta$ and $\lambda$. These bounds are given in terms of the order $n$ and the diameter $D$ of the graph as it is usual for locating and dominating parameters. Similar studies can be found in~\cite{exlara07} for identifying codes, and in~\cite{chahulo06} when the action range of a locating-dominating code is $r>1$.

\begin{thm}\label{boundseta} Let $G$ be a
graph such that $|V(G)|=n$, $diam(G)=D\ge3$ and $\eta(G)=\eta$. Then
$\eta + \lceil \frac{2D}{3}\rceil \le  n \le  \eta + \eta \cdot 3^{\eta-1}$, and both bounds are tight.
\end{thm}
\begin{proof}
Let $\cal P$ be a diameter joining a diametral pair $u$ and $v$ with vertices $V({\cal P})=\{u=0,1,\ldots,v=D\}$. If  $A=\{1,4,7,\ldots,\min\{D,3\lceil \frac{D+1}{3}\rceil-2\}\}$, then the set $S=\{V(G)\setminus V({\cal P})\}\cup A$ has $n-\lceil \frac{2D}{3} \rceil$ elements and it is clearly dominating and locating. Hence, $\eta \le n-\lceil \frac{2D}{3} \rceil$. Moreover, the lower bound $\eta + \lceil \frac{2D}{3}\rceil$ is tight since $\eta(P_n)=\lceil \frac{n}{3} \rceil$ for every $n>3$.

To prove the upper bound, consider an $\eta$-code $S=\{v_1,\ldots,v_{\eta}\}$ and an arbitrary vertex $u\in V(G)\setminus S$. As $S$ is a dominating set then, for some vertex $v_i\in S$, $d(u,v_i)=1$. Let $v_j\in S$ where $j\neq i$. Clearly $|d(v_i,v_j)-d(u,v_j)|\leq 1$, since $d(u,v_j)\le d(u,v_i)+d(v_i,v_j)=1+d(v_i,v_j)$ and $d(v_i,v_j)\le d(v_i,u)+d(u,v_j)=1+d(u,v_j)$. This means that the cardinality of $\displaystyle \{c_S(v)\}_{v\in V\setminus S}$ is at most $\eta\cdot 3^{\eta-1}$. In other words, $n\leq \eta + \eta \cdot 3^{\eta-1}$.

Finally, we prove that this upper bound is  tight. Let $\eta\ge 2$ be and consider the $\eta$-dimensional strong grid $P_5^{\eta}=P_5\boxtimes\cdots\boxtimes P_5$. Let $G_{\eta}$ denote the induced subgraph of $P_5^{\eta}$, whose vertex set is $V(G_{\eta})=\displaystyle \bigcup_{i=0}^{\eta}A_i$, where:
\begin{itemize}
\item $A_0=\{v_1=(0,3,\ldots,3),\ldots,v_{\eta}=(3,\ldots,3,0)\}$,
\item For every $i\in\{1,\ldots,\eta\}$,\\
$ A_i=\{(x_1,\ldots,x_{\eta}) ~|~ x_i=1,\text{and} ~j\neq i \Rightarrow 2\le x_j \le 4\}$
\end{itemize}
So the order of $G_{\eta}$ is $\eta + \eta \cdot 3^{\eta-1}$. It is easy to check that $A_0$ is an $\eta$-code of this graph.
\end{proof}

In \cite{chahulo07}, it is proved that $n  \le  \lambda (G) +2^{\lambda(G)}-1$ in any graph $G$ of order $n$ and it is a tight bound. In the following result we provide a lower bound which turns out to be also tight.

\begin{thm}\label{boundslambda} Let $G$ be a
graph of order $n$, diameter $D\ge 3$ and $\lambda(G)=\lambda$. Then
$\lambda + \lceil \frac{3D-1}{5} \rceil \le  n  $, and the bound is tight.
\end{thm}
\begin{proof}
Let $u,v\in V(G)$ two diametral vertices and let $\cal P$ be the diameter joining them. If $V({\cal P})=\{u=0,1,\ldots,v=D\}$ and $D=5h+k$ with $0\le k\le4$, then it is easy to check that, for $0\le k\le1$ (resp. $2\le k\le4$), the set $A=\{1,3,\ldots,5h-4,5h-2,D\}$ (resp. $A=\{1,3,\ldots,5h-4,5h-2,5h+1,D\}$) has $\lceil \frac{2D+2}{5} \rceil$ elements and it is a $\lambda$-set of $\cal P$. In other words,
the set $S=\{V(G)\setminus V({\cal P})\}\cup A$ has $n-\lceil \frac{3D-1}{5} \rceil$ elements and it  is a locating-dominating set of $G$. Hence, $\lambda(G) \le n-\lceil \frac{3D-1}{5} \rceil$. Moreover, the lower bound $\lambda(G) + \lceil \frac{3D-1}{5}\rceil$ is tight since, for every $n>3$, $\lambda(P_n)=\lceil \frac{2n}{5} \rceil$.

\end{proof}

An interesting case occurs when the graph is a tree $T$ of order $n$ (see~\cite{ldic_trees,blchloma11}).  A vertex of degree 1 is called a \emph{leaf}, a vertex adjacent to a leaf is a \emph{support vertex}, and if a vertex is adjacent to, at least, two leaves then it is called a \emph{strong support vertex}. The number of leaves and support vertices are denoted $l(T)$ and $s(T)$ respectively.

In \cite{heoe}, it was proved that there is no constant $k$ such that $\lambda(G)\le k \eta(G)$, for every graph $G$. However, it is also showed that $\lambda(T)\le 2 \eta(T)$ for every tree $T$ and that $\eta(T)=\gamma(T)+l(T)-s(T)$. Going a step further, we obtain the following result which turns out to give tight bounds.

\begin{thm}\label{treesetalambda} Let $T$ be a
tree of order at least 3, different from $P_6$ such that $\eta(T)=\eta$ and $\lambda(T)=\lambda$. Then
$\eta \le \lambda \le 2\eta-2$, and both bounds are tight.
\end{thm}
\begin{proof} If $T$ is the star $K_{1,n-1}$, then $\eta=\lambda=n-1$. Assume thus that $T$ is a tree of order $n\ge4$ and diameter $D\ge3$, and proceed by induction on $n$. Certainly, the statement is true for every tree of order at most 4. By hypothesis of induction, assume that it is also true for any tree of order less or equal than $n-1$ and let $T$ be a tree of order $n$. We distinguish two cases:

Case 1: There is no strong support vertex in $T$. Let $x,y$ be a diametral pair of leaves, and let $z$ be the support vertex of $y$, which clearly satisfies $deg(z)=2$. Let $u$ be the vertex adjacent to $z$ and different from $y$ in the diameter joining $x$ and $y$. Again, we have three subcases:

\begin{itemize}
\item Suppose that $deg(u)=2$ and consider the tree $T'=T-\{u,z,y\}$. Observe that $\eta(T)=\gamma(T)=\gamma(T')+1=\eta(T')+1$ and  $\lambda(T)\le\lambda(T')+2$. Hence,
$\lambda(T)\le \lambda(T')+2 \le 2\eta(T') -2 +2 = 2(\eta(T)-1)=2\eta(T)-2$.

\item Assume that $u$ is a support vertex of $T$ such that $deg(u)\ge3$ and consider the tree $T'=T-\{z,y\}$. Then $\eta(T)=\eta(T')+1$ and  $\lambda(T)\le\lambda(T')+1$. Hence,  $\lambda(T)\le \lambda(T')+1 \le 2\eta(T') -2 +1 = 2(\eta(T)-1)-1=2\eta(T)-3$.

\item Suppose  that $u$ is not a support vertex of $T$ and its degree is at least 3, which means that there exists a leaf $y'$, different from $y$,  adjacent to a support vertex $z'$ which is adjacent to $u$. We build the tree $T'=T-\{z,z',y,y'\}$. Note that $\eta(T')+1\le\eta(T)\le\eta(T')+2$ and  $\lambda(T)\le\lambda(T')+2$. Thus,  $\lambda(T)\le\lambda(T')+2 \le 2\eta(T') -2 +2 \le 2(\eta(T)-1)=2\eta(T)-2$.
\end{itemize}

Case 2: $T$ is a tree with at least one strong support vertex $w$. Consider the tree $T'=T-\{y\}$, where  $y$ is a leaf adjacent to $w$. Notice that $\eta(T)=\eta(T')+1$ and  $\lambda(T)=\lambda(T')+1$. Hence,  $\lambda(T)= \lambda(T')+1 \le 2\eta(T') -2 +1 = 2(\eta(T)-1)-1=2\eta(T)-3$.

Those bounds are tight since they are attained in the families of spiders $S_{k,3}$  and $S_{k,4}$ (see Figure~\ref{spiders}). Notice that $\{b_r\}_{r=1}^{k}\cup\{x\}$ is both an $\eta$-code and a $\lambda$-code of $S_{k,3}$, and observe also that $\{c_r\}_{r=1}^{k}\cup\{x\}$ and $\{a_r\}_{r=1}^{k}\cup\{c_r\}_{r=1}^{k}$  are an $\eta$-code and a $\lambda$-code of $S_{k,4}$, respectively.
\end{proof}

\begin{figure}
\begin{center}
\includegraphics[width=0.25\textwidth]{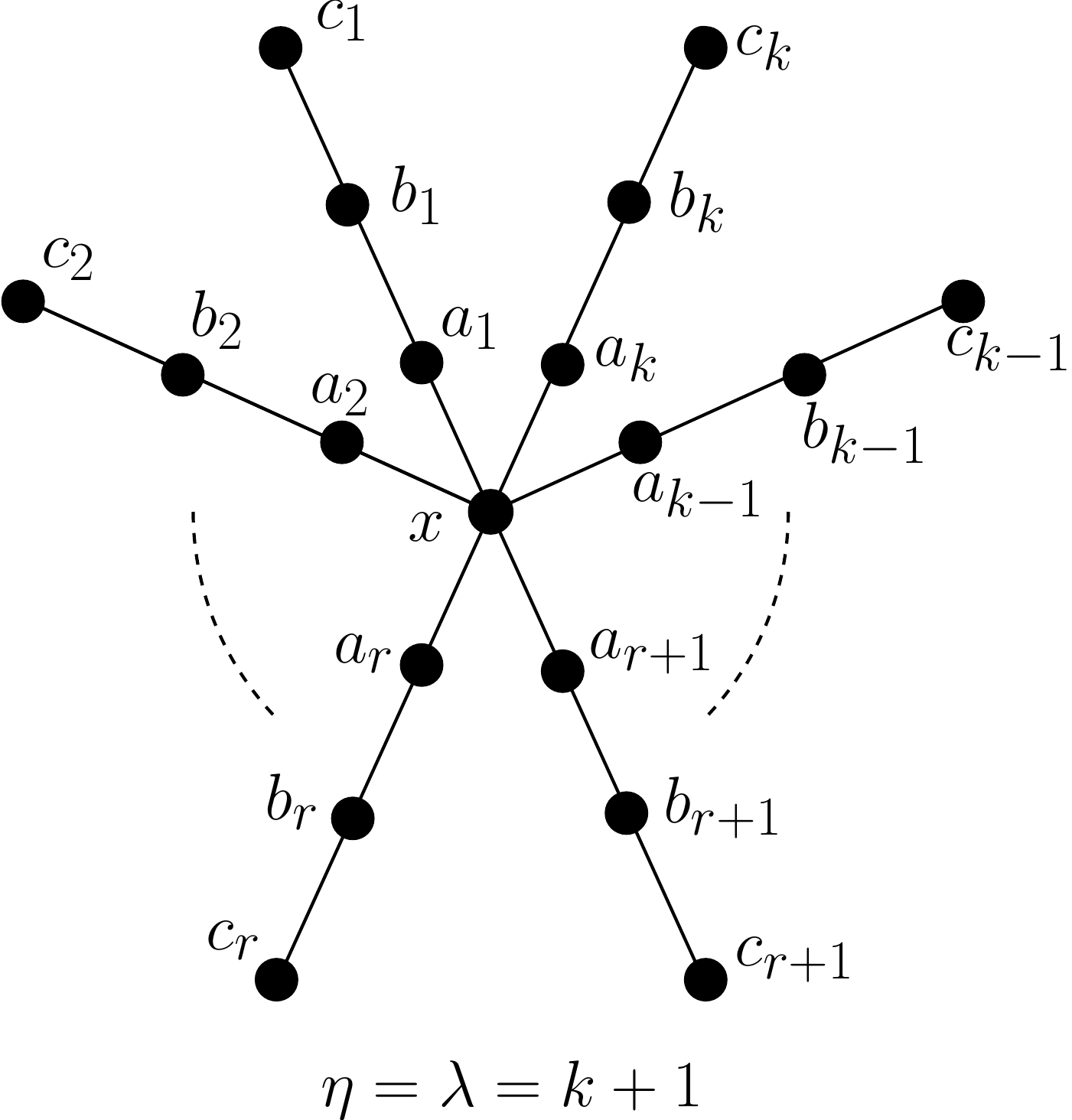}
\vspace{0.5cm}

\includegraphics[width=0.35\textwidth]{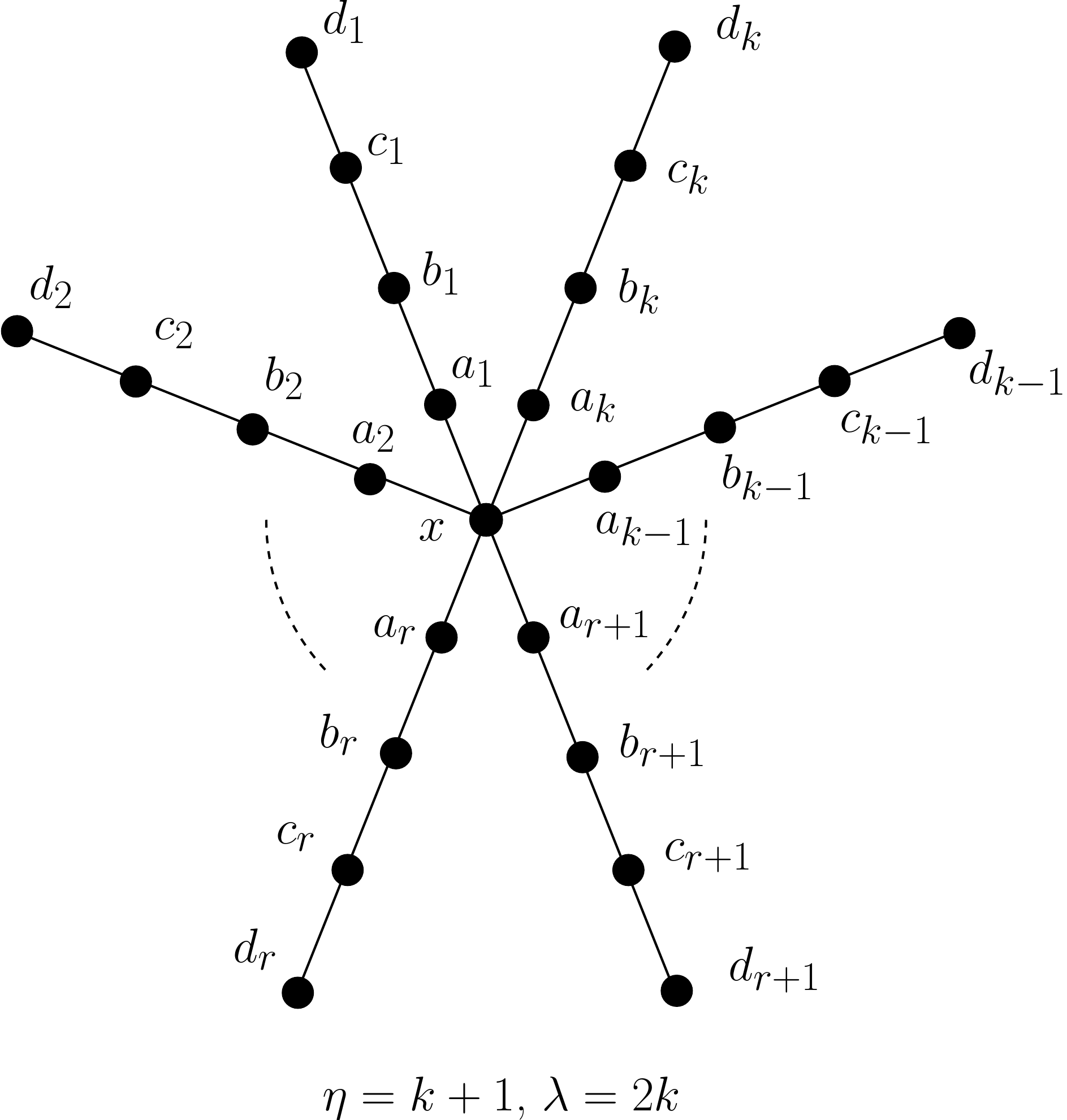}
\end{center}
 \caption{Spiders $S_{k,3}$  and $S_{k,4}$ on $k$ legs, all of them having 3 and 4 edges, respectively.}
\label{spiders}
\end{figure}

\section{Extremal values}\label{valoresextremos}
This section is devoted to establish sufficient conditions over an arbitrary graph $G$ which guarantee some extremal values for $\eta(G)$ and $\lambda(G)$. Graphs with order $n$ and $\eta$ or $\lambda$ equal to $1$ or $n-1$ have been characterized, as well as those graphs with $\eta=n-2$~\cite{heoe}. As a step further, we characterized here all the graphs with $\eta=2$, $\lambda=2$ and $\lambda=n-2$.

To begin with, the next result provides conditions for those graphs $G$ having $\eta(G)=\lambda(G)$.

\begin{prop} Let $G$ be a graph of order $n$, diameter $D$ and metric dimension $\beta$. If either $D=2$ or $\beta\ge n-3$, then $\eta(G)=\lambda(G)$.
\end{prop}
\begin{proof} Suppose that $D =2$ and let $S$ be an $\eta$-code of $G$. Since the maximum distance between vertices is 2, the vector of metric coordinates of $u$ with respect to $S$ contains only digits 1 and 2. Thus, $S$ must be also a $\lambda$-code.

On the other hand, suppose $\beta(G)\geq n-3$. In \cite{ours2}, it is proved that $\beta+D\le  n$ and all the graphs having metric dimension $n-3$ are given. Hence, the diameter of $G$ is at most 3. If $D=1$, then $G=K_n$ and thus $\eta(G)=\lambda(G)=n-1$. The case $D=2$ has been proved above and for the case $D=3$ it is straightforward to check that all the graphs provided in~\cite{ours2} satisfies $\eta(G)=\lambda(G)$.
\end{proof}

\begin{rem} \rm This result is tight in the sense that there are graphs having diameter greater than 2 and/or metric dimension less than $n-3$, satisfying  $\eta(G)<\lambda(G)$. For example, path $P_6$  verifies $diam(P_6)=5$, $\beta(P_6)=1$, $\eta(P_6)=2$ and $\lambda(P_6)=3$.
\end{rem}

Next, we characterize the family of graphs satisfying that $1\le\eta=\lambda\le2$. To begin with, it is clear that the unique graph with order $n\geq 2$ and $\eta=1$ is $P_2$, which certainly also satisfies $\lambda=1$. The case $\eta=2$ is mainly solved using the following results.

\begin{lem}\label{lemeta2} Let $G$ be a
graph of order $n$ and $\eta(G)=2$. Then,
\begin{enumerate}
\item[(i)] $3\le n\le 8$.
\item[(ii)] If $S=\{u,v\}$ is an $\eta$-code, then $d(u,v)\le3$.
\item[(iii)] The graph $G$ can be isometrically embedded into the king grid $P_5\boxtimes P_5$.
\end{enumerate}
\end{lem}
\begin{proof}
(i) These inequalities are obtained as a consequence of Theorem\,\ref{boundseta}, having also in mind that $\eta(K_2)=1$.

(ii) It is enough to realize that if $d(u,v)\ge4$, then $S$ is not a dominating set.

(iii) Let $G$ be a graph with $\eta(G)=2$ and let $S=\{ u,v\}$ be an $\eta$-code of $G$. Since $d(u,v)\le 3$ and every vertex of $G$ is adjacent either  to $u$ or to $v$, we have that $\{(d(u,x), d(v,x)):x\in V(G) \}$ is a subset of $[0,4]\times [0,4]$.
As $S$ is locating, we have $c_S(x)=(d(u,x), d(v,x))\not= (d(u,y), d(v,y))=c_S(y)$ for every pair of distinct vertices $x, y$, so there is an injection from $V(G)$ to $V(P_5\boxtimes P_5)$ simply by identifying every vertex $x$ of $G$ with its metric coordinates $(d(u,x), d(v,x))$ as a vertex in $P_5\boxtimes P_5$. Moreover, if two vertices $x$ and $y$ are adjacent in $G$, then
\begin{center}
$d(u,y)\le d(u,x)+d(x,y)=d(u,x)+1$ and $d(u,x)\le d(u,y)+d(y,x)=d(u,y)+1$,
\end{center}
which means that  $|d(u,x)-d(u,y)|\le 1$. Similarly, we obtain that $|d(v,x)-d(v,y)|\le 1$. Hence, $(d(u,x), d(v,x))$ and $(d(u,y), d(v,y))$ are adjacent in $P_5\boxtimes P_5$, therefore the above injection is an isometric embedding of $G$ in $P_5\boxtimes P_5$.
\end{proof}

\begin{figure}[ht]
  \begin{center}
        \includegraphics[width=0.45\textwidth]{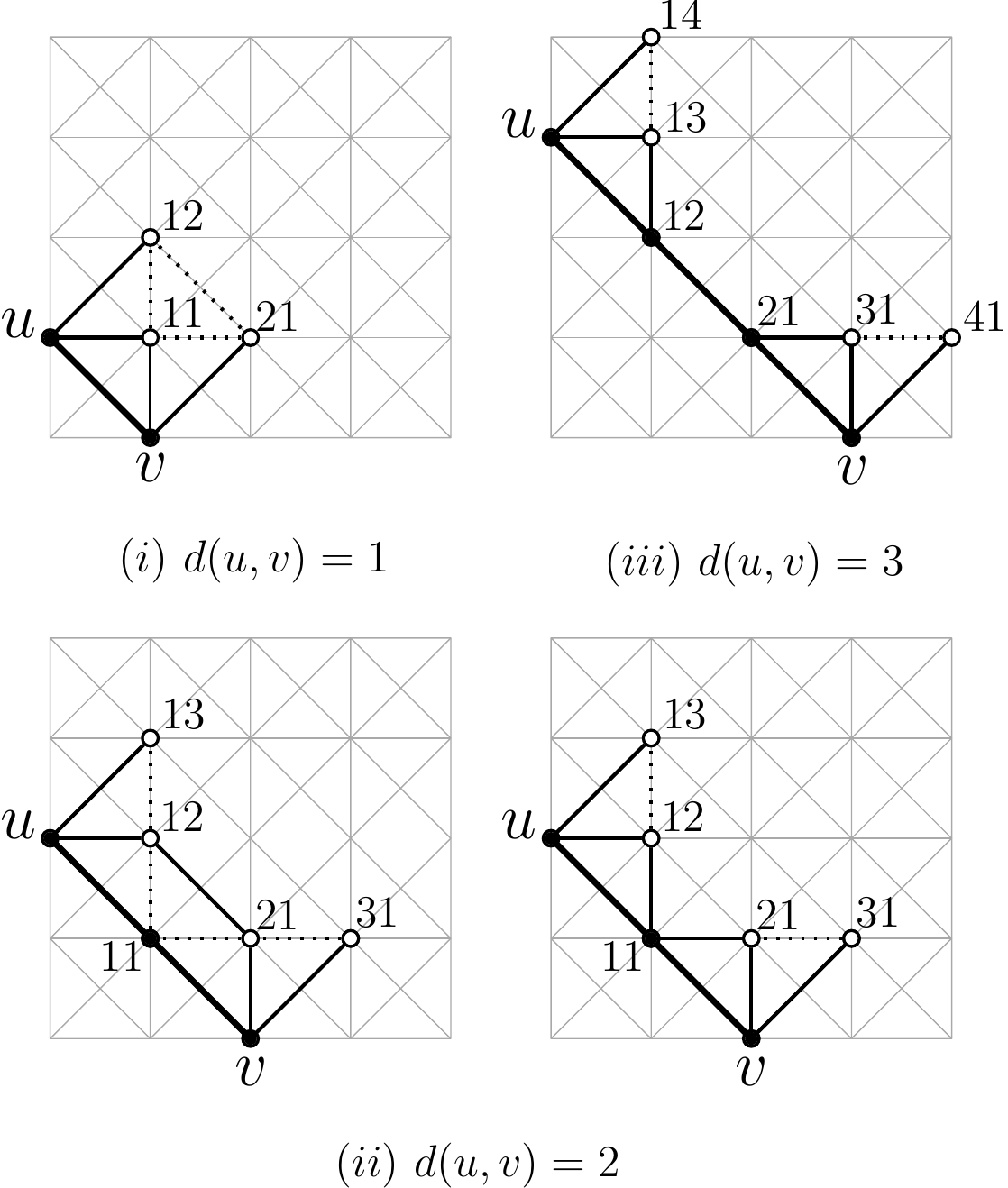}
  \end{center}
  \caption{Graphs verifying $\eta(G)=2$ isometrically embedded in $P_5\boxtimes P_5$. White vertices are optional, provided that the graph has at least three vertices. Discontinuous edges are optional and black vertices are compulsory. }
  \label{figuraEta2P5P5}
\end{figure}

\begin{thm}
There exist 51 non-isomorphic graphs satisfying $\eta(G)=2$ (see Figure~\ref{figuraetados}).
\end{thm}
\begin{proof}
Let $S=\{u,v\}$ be an $\eta$-code of $G$. We label every vertex $w\in V(G)$ with the pair of integers $(d(u,w), d(v,w))$. According to Lemma\,\ref{lemeta2}, $3\le n\le 8$ and $1\le d(u,v)\le3$.
We distinguish three cases, depending on the distance between vertices $u$ and $v$.

{\bf Case 1.}  If $d(u,v)=1$ then $V(G)\setminus S\subseteq \{(1,1),(1,2),(2,1)\}$, and hence $3\le n \le5$. Following a similar reasoning as in the proof of Lemma~\ref{lemeta2}, we obtain that $G$ can be isometrically embedded into the king grid $P_3\boxtimes P_3$, as showed in Figure~\ref{figuraEta2P5P5}(i).

{\bf Case 2.} Suppose $d(u,v)=2$. Then $V(G)\setminus S\subseteq \{(1,1),(1,2),(1,3),(2,1),(3,1)\}$, so $3\le n\le 7$. Again using the injection defined in the proof of Lemma\,\ref{lemeta2}, $G$ can be embedded isometrically into $P_4\boxtimes P_4$ (see Figure~\ref{figuraEta2P5P5}(ii)).

{\bf Case 3.} Finally, if $d(u,v)=3$, now $V(G)\setminus S\subseteq \{(1,2),(1,3),(1,4),(2,1),(3,1),(4,1)\}$ and $4\le n \le 8$. Lemma~\ref{lemeta2} gives us again the isometric embedding of $G$ into $P_5\boxtimes P_5$ (see Figure~\ref{figuraEta2P5P5}(iii)).

An exhaustive inspection of all possibilities proves that the set of non-isomorphic graphs satisfying $\eta(G)=2$ has order 51, showed in Figure~\ref{figuraetados}, and consists of two graphs of order 3, four graphs of order 4, ten graphs of order 5, fifteen graphs of order 6, seventeen graphs of order 7, and three graphs of order 8.
\end{proof}

\begin{figure}[ht]
  \begin{center}
        \includegraphics[width=0.45\textwidth]{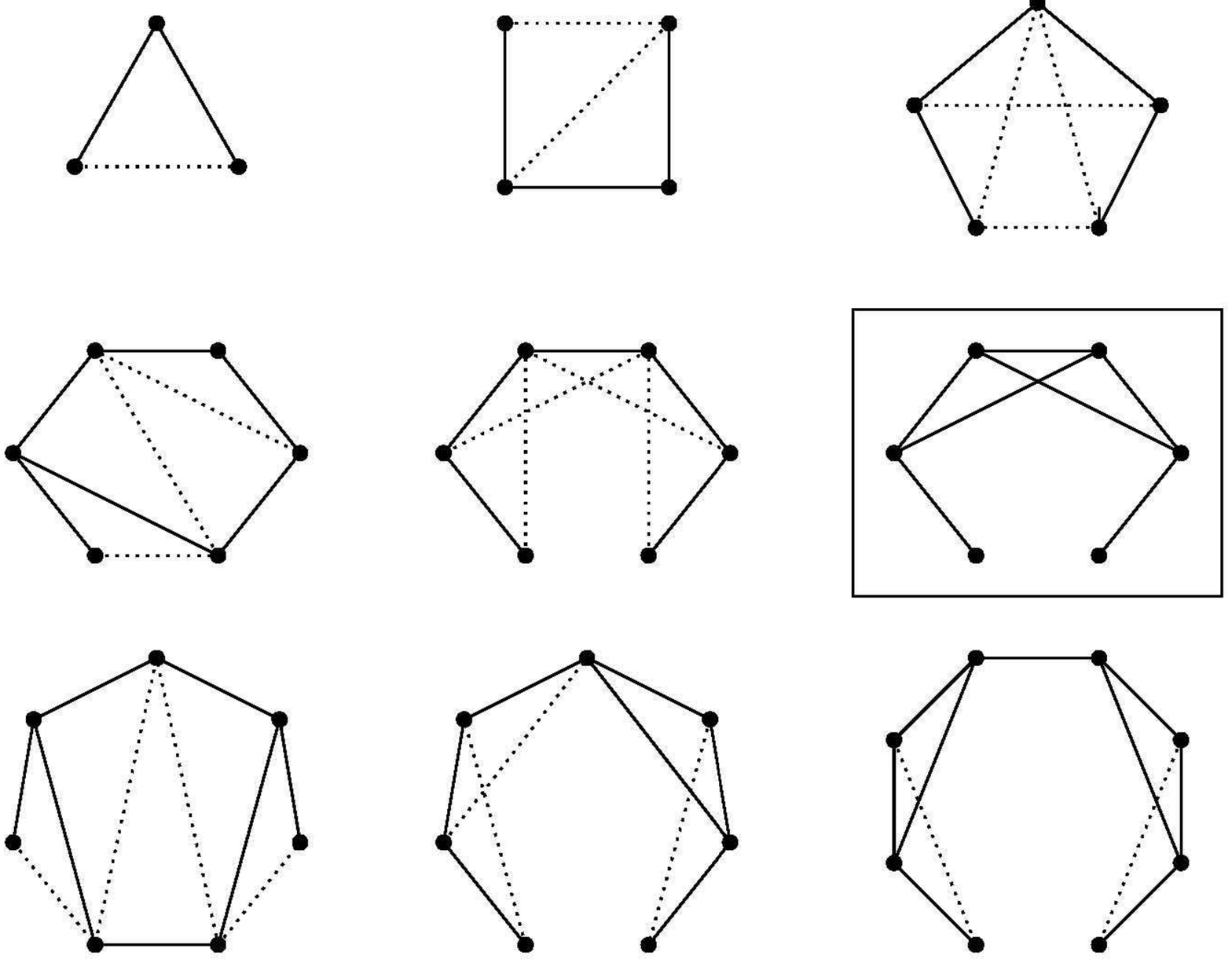}
  \end{center}
  \caption{All graphs satisfying $\eta(G)=2$. Discontinuous edges are optional and the framed graph is forbidden. The three graphs in the first row are all the graphs having $\lambda=2$.}
  \label{figuraetados}
\end{figure}

As a consequence of the previous theorem, it is also possible to obtain a similar list of graphs for $\lambda$.
 \begin{cor}
There are 16 non-isomorphic graphs satisfying $\lambda(G)=2$ (see Figure\,\ref{figuraetados}).
\end{cor}
\begin{proof}
Let $G$ be a graph of order $n$ satisfying $\lambda(G)=2$. From Proposition\,\ref{ineq} and Theorem\,\ref{boundslambda} it is immediately derived that $3\le n\le 5$ and $\eta(G)=2$, which means that $G$ must be one of the graphs displayed in the first row of Figure\,\ref{figuraetados}. For finishing  the proof, it is enough to check that each of these 16 graphs satisfies $\lambda(G)=2$.
\end{proof}

We end up this section by characterizing  the family of graphs for which $n-2\le\eta(G)=\lambda(G)\le n-1$. In \cite{heoe} (resp. \cite{slater88}) , it was proved that if $G$ is a graph such that  $\eta(G)=n-1$ (resp. $\lambda(G)=n-1$), then $G$ is either the complete graph $K_n$ or the star $K_{1,n}$. Also in \cite{heoe} , all graphs $G$  such that $\eta(G)=n-2$ were completely characterized. As a consequence, all these graphs must also fulfill  that $\lambda(G)=n-2$. Next, we show that these are the unique graphs satisfying the equation $\lambda(G)=n-2$.

\begin{lem}\label{D3} Let $G$ be a graph with diameter $D$, order $n$ and $\lambda(G)\ge n-2$. Then $D\leq3$.\end{lem}

\begin{proof} Suppose that $D\ge4$ and take $u,v\in V(G)$ such that $d(u,v)=4$. If ${\cal P}$ is a shortest path joining $u$ and $v$ such that $V({\cal P})=\{u,a,w,b,v\}$, then it is straightforward to check that the set $V(G)\setminus\{u,w,v\}$ is locating-dominating.
\end{proof}

\begin{thm}\label{lambda=n-2} Let $G$ be a graph of order $n\ge3$. Then, $\lambda(G)=n-2$ if and only if $\eta(G)=n-2$.
\end{thm}
\begin{proof} Since $\eta(G)=n-1$ if and only if $\lambda(G)=n-1$  and $\eta(G)\le\lambda(G)$, it is clear that $\eta(G)=n-2$ implies that $\lambda(G)=n-2$.

To prove the converse, assume on the contrary that there exists a graph $G$ with $\lambda(G)=n-2$ and $\eta(G)<n-2$. Let $S=V(G)\setminus \{ x,y,z\}$ be a metric-locating-dominating set of cardinality $n-3$. Since $S$ is not locating-dominating, suppose without loss of generality that $N(x)\cap S=N(y)\cap S=N\not= \emptyset$. However $N(x)\not=N(y)$ because $S$ is locating, that is, either $x$ or $y$ but not both must be adjacent to $z$. Assume hence that $yz\in E(G)$ and $xz\notin E(G)$ (see Figure~\ref{fign-2}(i)).

Since $S$ is a locating set, there exists a vertex $w\in S\setminus N$ such that $d(x,w)\not= d(y,w)$. Notice that no vertex in $N$ is adjacent to $w$, as otherwise $d(x,w)=d(y,w)=2$ (see Figure \ref{fign-2}(ii)). According to Lemma \ref{D3}, $2\leq d(x,w),d(y,w)\leq 3$ and so $d(y,w)=2$ and $d(x,w)=3$, since $N(w)\cap N(x)=\emptyset$. Moreover, it is also followed that $N(w)\cap N(y)=\{z\}$ (see Figure~\ref{fign-2}(ii)).

\begin{figure}[ht]
\includegraphics[width=0.45\textwidth]{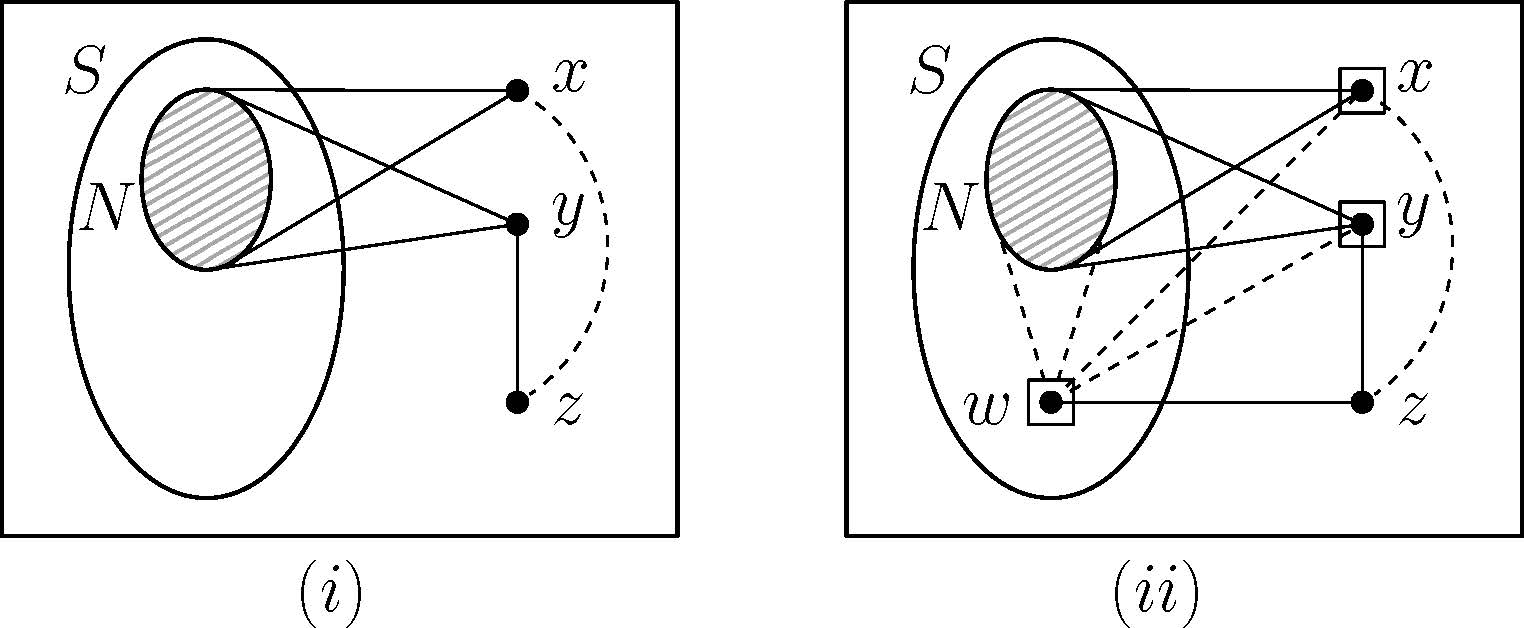}
\caption{Solid edges join  adjacent vertices and dashed edges join non-adjacent vertices.}
\label{fign-2}
\end{figure}

Finally, consider the set $S'=V\setminus \{ x,y,w \}$ and note that
  \begin{itemize}
  \item $N(x)\cap S'=N\not= \emptyset$,
  \item $N(y)\cap S'=N\cup \{ z \}\not= \emptyset$,
 \item $z\in N(w)\cap S'\not= \emptyset$ and $N \nsubseteq N(w)\cap S'$.
 \end{itemize}
In other words, $N(x)\cap S'$, $N(y)\cap S'$ and $N(w)\cap S'$ are pairwise different and non-empty. Therefore, $S'$ is a locating-dominating set of cardinality $n-3$, which leads to a contradiction.
\end{proof}

\begin{rem}\rm As a  consequence of the above result it is immediately concluded that $\eta(G)=n-3$ implies $\lambda(G)=n-3$.  However, the reciprocal is not true. For example, the path $P_6$  verifies $\eta(P_6)=2=n-4$ and $\lambda(P_6)=3=n-3$.
\end{rem}

\begin{rem}\rm As showed in~\cite{heoe}, graph families satisfying $\eta=n-2$ are the following:
\begin{itemize}
\item $K_{r,s}$,  the complete bipartite graph, $r,s\ge 2$,
\item $K_r+\overline{K_s}$, $r,s\ge 2$,
\item $K_1+(K_r\cup \overline{K_s})$, with $r,s\ge 2$,
\item $K_r+(K_1\cup K_s)$, with $r\ge 1$, $s\ge 2$,
\item $K_2(r,s)$, the double star, that is $r$ and $s$ pendant vertices from the two vertices of $K_2$,
\item $K_1+(K_{1,r}\cup  \overline{K_s})$, with $r \ge 2$, $s\ge 1$,
\item Any graph obtained by adding a new vertex adjacent to $s$ leaves of the star $K_{1,r}$, $2\le s\le r-1$.
\end{itemize}
As a consequence of Theorem~\ref{lambda=n-2}, those are also the graphs with $\lambda=n-2$.
\end{rem}
\section{Realization theorems}\label{realizabilidad}

In this section, we characterize when it is possible to construct examples for a variety of values for $\eta$, $\lambda$, $\beta$ and $\gamma$.

\begin{thm}\label{rt1}
Given three positive  integers $a,b, c$ verifying that $\max\{a,b\}\leq c\leq a+b$, there always exists a graph $G$ such that $\gamma(G)=a$,  $\beta(G)=b$ and $\eta(G)=c$, except for the case $1=b<a<c=a+1$.
\end{thm}
\begin{proof} We distinguish different cases:

{\bf Case 1.}  Suppose that $b=1$. Certainly, $\beta(G)=1$ if and only if $G$ is a path $P_n$. Moreover, according to Table\,\ref{tabla},  $\gamma(P_2)=\beta(P_2)=\eta(P_2)=1$, $\gamma(P_3)=\beta(P_2)=1<2=\eta(P_2)$ and $\beta(P_{n})=1<\gamma(P_{n})=\eta(P_{n})=\lceil \frac{n}{3}\rceil$ whenever $n\ge4$. Hence, $P_2$ satisfies case $a=b=c=1$, $P_3$ fulfill the case $1=a=b<c=a+b=2$. The case $1=b<a<c=a+1$ is not realizable, and for every $k\ge2$, $P_{3k}$ verifies the case $1=b<a=c=k$.

{\bf Case 2.} Suppose now that $a=1$ and $ b\geq 2$. Notice that if $\gamma(G)=1$, then
 $\beta(G)\leq \eta(G) \le \beta(G)+1$ and moreover there exists a vertex which is adjacent to the rest of vertices of the graph. So the case $1=a<b=c$ is achieved by considering the complete graph $K_{b+1}$, and the case $1=a<b<c=b+1$ is realized with the star $K_{1,b+1}$.

{\bf Case 3.} Finally if $a\geq 2$ and $b\geq 2$ we need to consider several subcases. Recall that in every case
$c\leq a+b$.

{\bf Case 3.1.} When $2\leq a\leq b=c$,
the graph showed in Figure \ref{case3}(i) realizes this case for $r=a-1\geq 1$ and $l=b-a+1\geq 1$. Note that
$\{x_i  \}_{i=1}^r  \cup \{ w \}$ is a $\gamma$-code of cardinality $r+1=a$
and $\{ x_i  \}_{i=1}^r  \cup \{ \alpha_i  \}_{i=1}^l$ is both a $\beta$-code
and an $\eta$-code of cardinality $r+l=b=c$.

{\bf Case 3.2.} If $2\leq a=b<c$, then the graph displayed in Figure \ref{case3}(ii) does the work by taking $r=2a-c\geq 0$ and $s=c-a\geq 1$. It is straightforward  to prove that
$\{x_i  \}_{i=1}^r  \cup \{v_i  \}_{i=1}^s$ is a $\gamma$-code of cardinality $r+s=a$,
$\{x_i  \}_{i=1}^r  \cup \{z_i  \}_{i=1}^s$ is a $\beta$-code of cardinality $r+s=a=b$,
and $\{x_i  \}_{i=1}^r  \cup \{z_i  \}_{i=1}^s   \cup \{v_i  \}_{i=1}^s$ is an $\eta$-code of cardinality $r+2s=c$.

{\bf Case 3.3.} Let $2\leq a < b < c $.
Consider the graph displayed in Figure \ref{case3}(iii) and take $r=a+b-c\geq 0$, $s=c-b-1\geq 0$ and $l=b-a+1\geq 2$.
Notice that $r$ and $s$ are not both $0$, otherwise
$c=a+b=b+1$ implying that $a=1$, which is a contradiction.
Therefore,
$\{ x_i  \}_{i=1}^r  \cup \{ v_i  \}_{i=1}^s \cup \{ w \} $ is a
$\gamma$-code  of cardinality $r+s+1=a$,
$\{ x_i  \}_{i=1}^r  \cup \{ z_i  \}_{i=1}^s \cup \{ \alpha_i  \}_{i=1}^l$ is a $\beta$-code
of cardinality $r+s+l=b$,
and
$\{ x_i  \}_{i=1}^r  \cup \{ z_i  \}_{i=1}^s  \cup \{ \alpha_i  \}_{i=1}^l \cup \{ v_i  \}_{i=1}^s \cup \{ w \}$
is an $\eta$-code  of cardinality  $r+2s+l+1=c$.

{\bf Case 3.4.} For the case $2\leq b<a=c$, consider the graph in Figure \ref{case3}(iv) and
take $r=b-1\geq 1$ and $l=a-b \geq 1$. Then
$\{ x_i  \}_{i=1}^r  \cup \{ \delta \} $ is a $\beta$-code of cardinality $r+1=b$ and
$\{ x_i  \}_{i=1}^r  \cup \{ w_i  \}_{i=1}^l \cup \{ \delta \} $ is both a $\gamma$-code and an $\eta$-code of cardinality $r+l+1=a=c$.

{\bf Case 3.5.} When $2\leq b<a<c$, consider the graph in Figure \ref{case3}(v) and let
$r=a+b-c\geq 0$, $s=c-a-1\geq 0$ and $l=a-b+1\geq 2$. Notice that $r$ and $s$ are not both $0$, otherwise $c=a+1=a+b$ implying $b=1$, which is a contradiction. Then
$\{ x_i  \}_{i=1}^r  \cup \{ v_i  \}_{i=1}^s \cup \{ w_i \}_{i=1}^l $
 is a $\gamma$-code of cardinality $r+s+l=a$,
$\{ x_i  \}_{i=1}^r  \cup \{ z_i  \}_{i=1}^s \cup \{ \delta \}$
 is a $\beta$-code  of cardinality $r+s+1=b$ and
$\{ x_i  \}_{i=1}^r  \cup \{ z_i  \}_{i=1}^s  \cup \{ \delta \}  \cup \{ v_i  \}_{i=1}^s \cup \{ w_i \}_{i=1}^l $
 is an $\eta$-code of cardinality $r+2s+l+1=c$.
\end{proof}

\begin{figure}[ht]
  \begin{center}
        \includegraphics[width=0.45\textwidth]{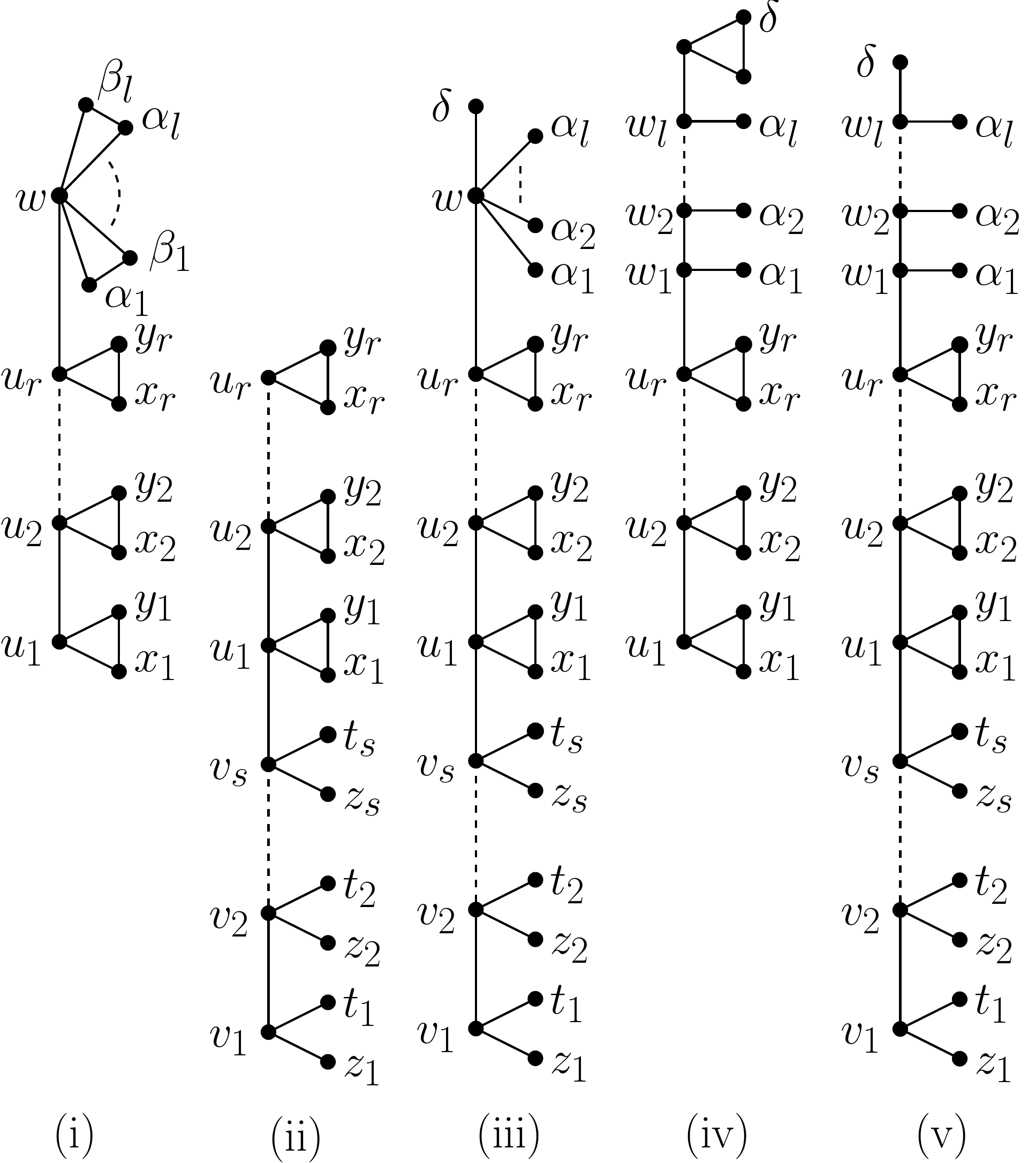}
 \caption{\label{case3} Cases $a\ge2$, $b\ge2$. }
  \end{center}

\end{figure}

\begin{figure}[ht]
  \begin{center}
        \includegraphics[width=.30\textwidth]{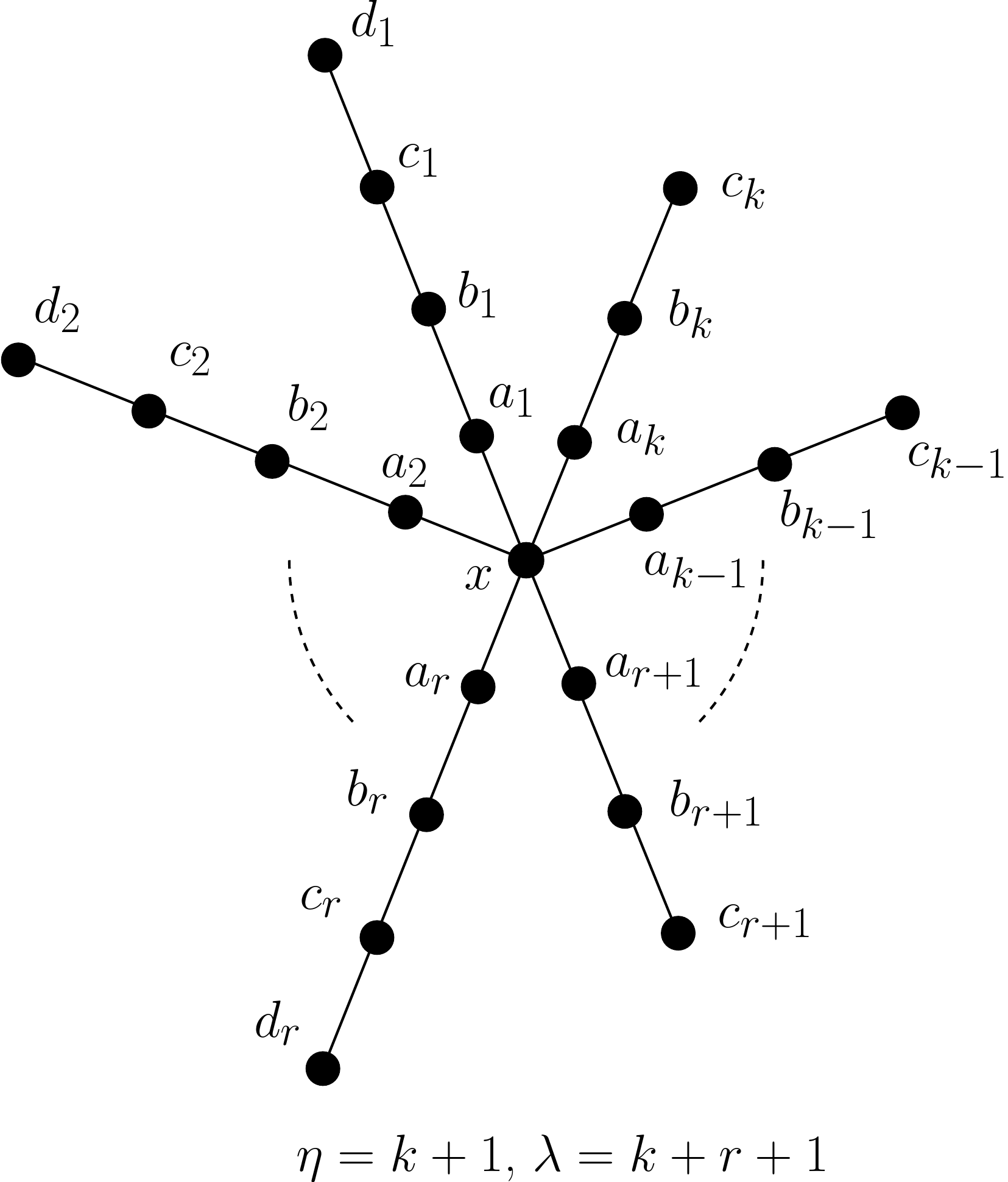}
  \end{center}
  \caption{Spider $S_{r,4,k-r,3}$ on $k$ legs, $r$ of them having 4 edges, and the rest 3 edges.}
  \label{spider1}
\end{figure}

Moreover, in the special case of trees, we can obtain the following result.

\begin{thm}\label{rt2}
Given two  integers $a$ and $b$ verifying that $3\le a \le b \le 2a-2$, there always exists a tree $T$ such that $\eta(T)=a$   and $\lambda(T)=b$.
\end{thm}
\begin{proof} Let $r,k$ be integers such that $2\le k$ and $0\le r \le k$. Consider the spider $S_{r,4,k-r,3}$ showed in Figure \ref{spider1}. Notice that $\{c_i\}_{i=1}^{r}\cup\{b_i\}_{i=r+1}^{k}\cup\{x\}$  is an $\eta$-code and  $\{c_i\}_{i=1}^{r}\cup\{a_i\}_{i=1}^{r}\cup\{b_i\}_{i=r+1}^{k}\cup\{x\}$   is a $\lambda$-code of  $S_{r,4,k-r,3}$. Hence, given any two integers $a,b$ such that $3\le a \le b \le 2a-2$, the spider $S_{b-a,4,2a-b-1,3}$ satisfies $\eta=a$  and $\lambda=b$.
\end{proof}



\end{document}